\documentclass[smallcondensed,a4paper]{svjour3}
\smartqed
\journalname{Geometriae Dedicata}

\usepackage[%
colorlinks=true,
linkcolor=black,
citecolor=black
]{hyperref}

\usepackage[slovak,american]{babel}

\usepackage{geometry}
\usepackage{paralist}
\usepackage{amsmath,amssymb} %,amsthm}
\usepackage{enumerate,units,graphicx,url}
\usepackage[percent]{overpic}
\graphicspath{{./figs/}}

\spnewtheorem{thm}{Theorem}{\bf}{\it}
\spnewtheorem*{holthm}{Holonomy Theorem}{\bf}{\it}
\spnewtheorem{lem}[thm]{Lemma}{\bf}{\it}
\spnewtheorem{cor}[thm]{Corollary}{\bf}{\it}
\spnewtheorem*{clm}{Claim}{\bf}{\rm}
\spnewtheorem*{dfn}{Definition}{\bf}{\rm}
\spnewtheorem*{rmk}{Remark}{\bf}{\rm}
\newcommand{\setm}{\smallsetminus}

\newcommand{\R}{{\mathbb R}}
\newcommand{\C}{{\mathbb C}}

\title{
There is no triangulation of the torus with \\
vertex degrees $5,6,\ldots,6,7$
and related results:\\ 
Geometric proofs for combinatorial theorems
}
\titlerunning{There is no 5,7-triangulation of the torus}

\author{Ivan Izmestiev \and 
  Robert B. Kusner \and 
  G\"unter Rote \and 
  Boris Springborn \and 
  John M. Sullivan}
\authorrunning{Izemestiev, Kusner, Rote, Springborn and Sullivan}

\institute{%
I.~Izmestiev \at
  Freie Universit\"at Berlin, Arnimallee 2, 14195 Berlin \\
  \email{ivan.izmestiev@gmail.com}
\and
R.~B.~Kusner \at
  University of Massachusetts, Amherst, MA 01003-9305 \\
  \email{kusner@math.umass.edu}
\and
G.~Rote \at
  Freie Universit\"at Berlin, Takustr. 9, 14195 Berlin \\
  \email{rote@inf.fu-berlin.de}
\and
B.~Springborn \at
  Technische Universit\"at M\"unchen, Boltzmannstr. 3, 85748 Garching \\
  \email{boris.springborn@tum.de}
\and
J.~M.~Sullivan \at
  Technische Universit\"at Berlin, Str. des 17. Juni 136, 10623 Berlin \\
  Tel.: +49-30-314-29279 \\
  Fax.: +49-30-314-29260 \\
  \email{sullivan@math.tu-berlin.de}
}

\begin{document}

\date{10 August 2012; revised 27 August 2012}

\maketitle

\begin{abstract}
  There is no $5,\!7$-triangulation of the torus, that is, no
  triangulation with exactly two exceptional vertices, of degree~$5$
  and~$7$. Similarly, there is no $3,\!5$-quadrangula\-tion. The vertices
  of a $2,\!4$-hexangulation of the torus cannot be bicolored. Similar
  statements hold for $4,\!8$-triangulations and
  $2,\!6$-quadrangulations. We prove these results, of which the first
  two are known and the others seem to be new, as corollaries of a theorem
  on the holonomy group of a euclidean cone metric on the torus with
  just two cone points. We provide two proofs of this theorem: One
  argument is metric in nature, the other relies on the induced
  conformal structure and proceeds by invoking the residue theorem.
  Similar methods can be used to prove a theorem
  of Dress on infinite triangulations of the plane with exactly two
  irregular vertices. The non-existence results for torus
  decompositions provide infinite families of graphs which cannot be
  embedded in the torus.
%  \subclass{05C07 \and 05C10 \and 30F10 \and 30F30 \and 52B05 \and
%  52C20 \and 57M12 \and 57M15 \and 57M50}
\end{abstract}

\section{Introduction}

In any triangulation of the torus, the average vertex degree is~$6$,
so vertices of degree $d\ne6$ can be considered \emph{exceptional}.
It is easy to find \emph{regular} triangulations with no exceptional vertices,
as in Figure~\ref{fig:reg-examples}.
\begin{figure}
\centering
(a)\hspace{-1em}\raisebox{-40pt}{\includegraphics[scale=.6]{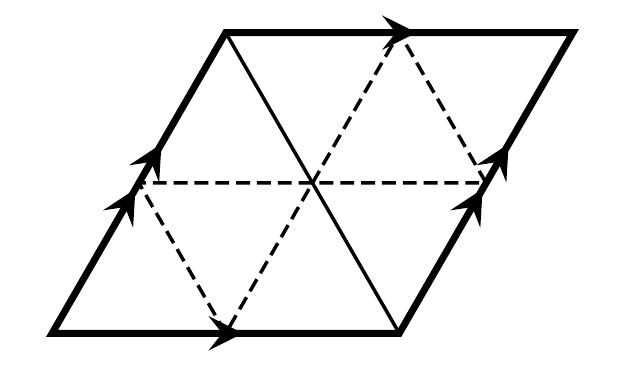}}\hspace{2em}%
(b)\hspace{-1em}\raisebox{-40pt}{\includegraphics[scale=.6]{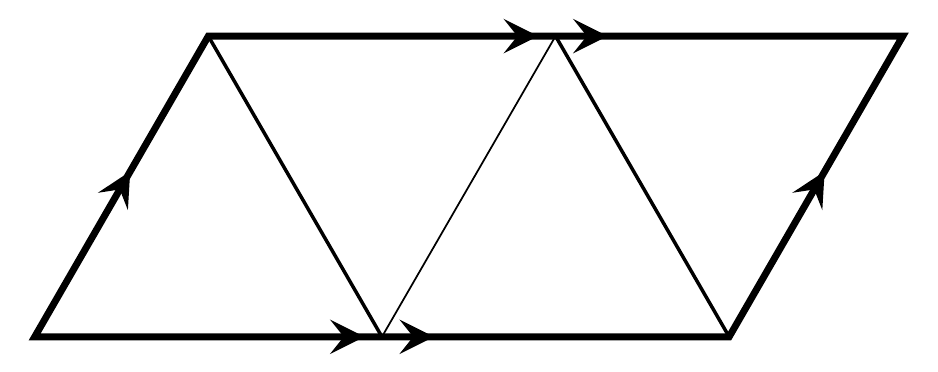}}\hspace{1.8em}%
(c)\raisebox{-50pt}{\includegraphics[scale=.4]{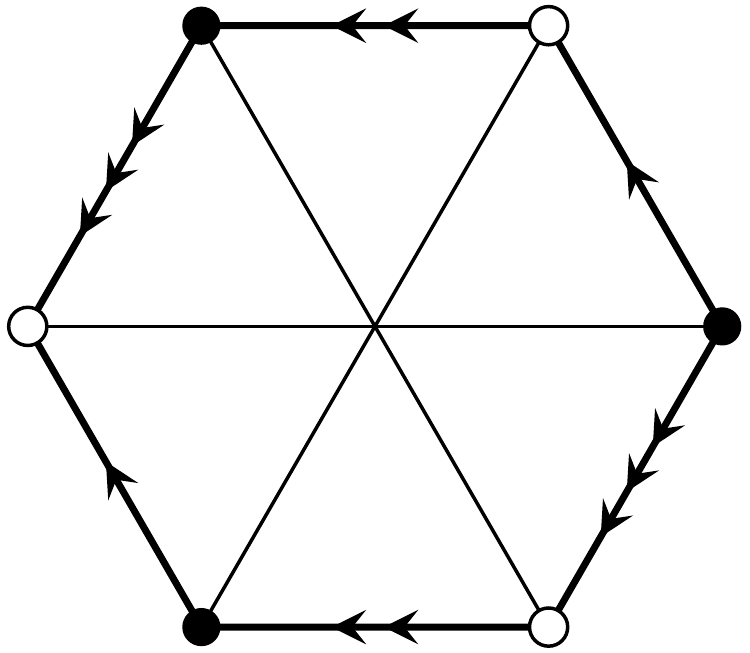}}
\caption[Regular triangulations of a torus]
{The simplest regular triangulations of the torus.
(a)~The solid lines show a triangulation with a single vertex
of degree~$6$.  Any regular triangulation is a cover
of this one.  After 
each triangle is split into four (dashed
lines) there are
four vertices.
(b)~The unique regular triangulation with two vertices.
(c)~There are two regular triangulations with three vertices, one
analogous to~(b),
and this more symmetric one. 
}\label{fig:reg-examples}
\end{figure}
Applying a single edge flip to such a triangulation produces a
triangulation with four exceptional vertices (assuming the four
vertices in question are distinct): two of degree~$5$ and two of
degree~$7$, as in Figure~\ref{fig:examples57}(a).  We call this a
$5^27^2$-triangulation.
\begin{figure}
  \centering
  (a)\hspace{-1em}%
  \raisebox{-40pt}{\includegraphics[scale=.6]{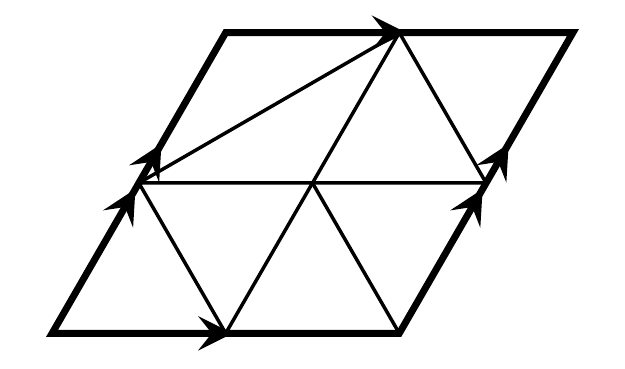}}
  \hspace{2cm}%
  (b)
  \raisebox{-45pt}{\includegraphics[scale=.6]{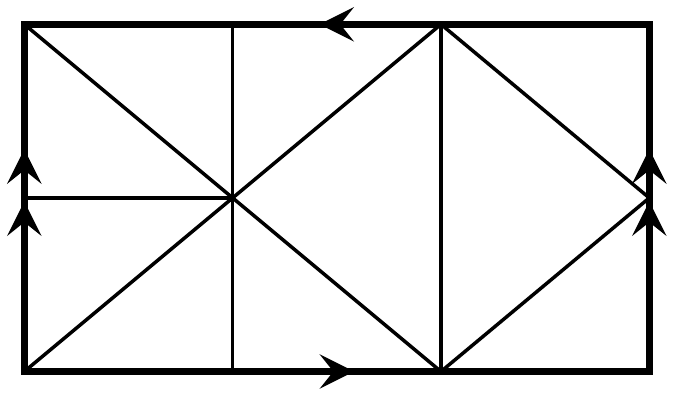}}
\caption[Triangulations with valence five and seven]
{(a)~Flipping one edge in the refined triangulation
of Figure~\ref{fig:reg-examples}(a) gives a $5^27^2$-triangulation
which has only the four exceptional vertices.
(b)~The Klein bottle does have $5,\!7$-triangulations,
for example this one with five vertices.}\label{fig:examples57}
\end{figure}
Similarly, we can produce examples
of triangulations with just two exceptional vertices, assuming these
have degrees other than~$5$ and~$7$. Figure~\ref{fig:examples}
shows $4,\!8$-, $3,\!9$-, $2,\!10$- and $1,\!11$-triangulations
of the torus.
\begin{figure}
  \centering
  (a)
  \raisebox{-40pt}{\includegraphics[scale=.6]{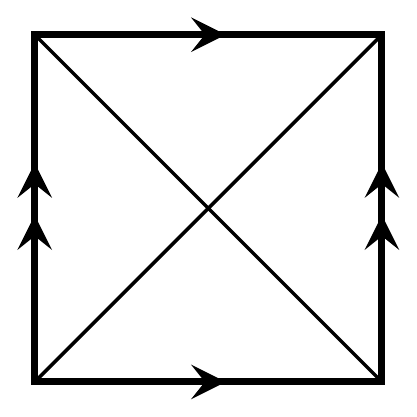}}
  \hspace{2em}(b)\hspace{-2em}
  \raisebox{-37pt}{\includegraphics[scale=.6]{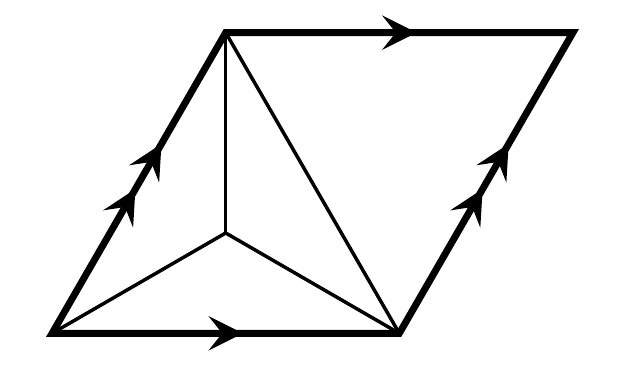}}
  \hspace{1em}
  (c)
  \raisebox{-40pt}{\includegraphics[scale=.6]{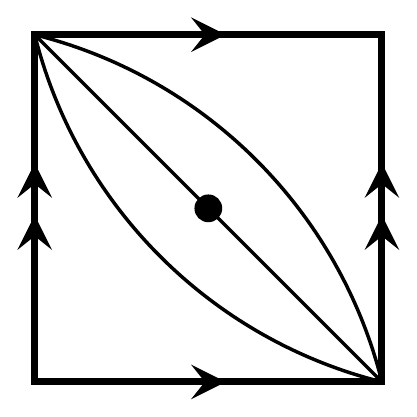}}
  \hspace{1em}
  (d)
  \raisebox{-40pt}{\includegraphics[scale=.6]{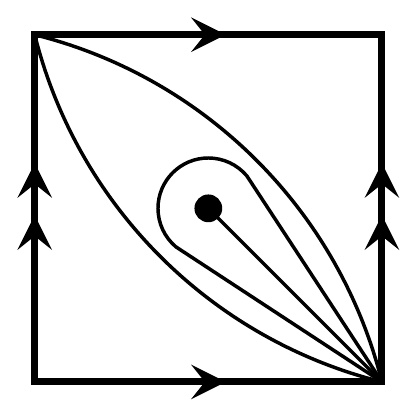}}
\caption[Example triangulations of a torus]
{The irregular triangulations of the torus with exactly two vertices:
(a)~a $4,\!8$-triangulation,
(b)~a $3,\!9$-triangulation,
(c)~a $2,\!10$-triangulation and
(d)~a $1,\!11$-triangulation.
}
\label{fig:examples}
\end{figure}
However:

\begin{thm}[\foreignlanguage{slovak}{Jendro^l \& Jucovi^c}~\cite{JeJu}]
  \label{thm:57}
  The torus has no $5,\!7$-triangulation, that is, no triangulation with
  exactly two exceptional vertices, of degree $5$ and $7$.
\end{thm}

We can also consider quadrangulations of the torus. 
In this case, the average vertex degree is $4$, and an analogous
theorem holds:

\begin{thm}[Barnette, \foreignlanguage{slovak}{Jucovi^c} \& Trenkler~\cite{BJT}]
  \label{thm:35}
  The torus has no $3,\!5$-quad\-ran\-gu\-la\-tion, that is, no
  quadrangulation with exactly two exceptional vertices, of degree~$3$
  and~$5$.
\end{thm}

On the other hand, $2,\!4$- and $3^25^2$-quadrangulations do exist,
as shown in Figure~\ref{fig:quadrs}.
\begin{figure}
  \centering
  (a)
  \raisebox{-40pt}{\includegraphics[scale=.6]{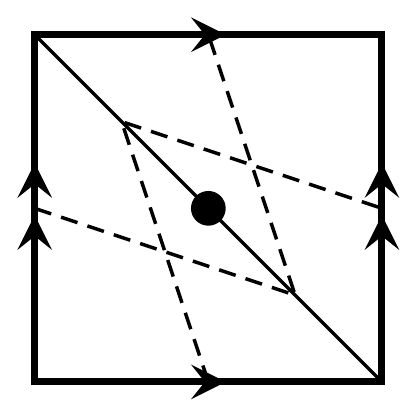}}
  \qquad\qquad
  (b)
  \raisebox{-40pt}{\includegraphics[scale=.6]{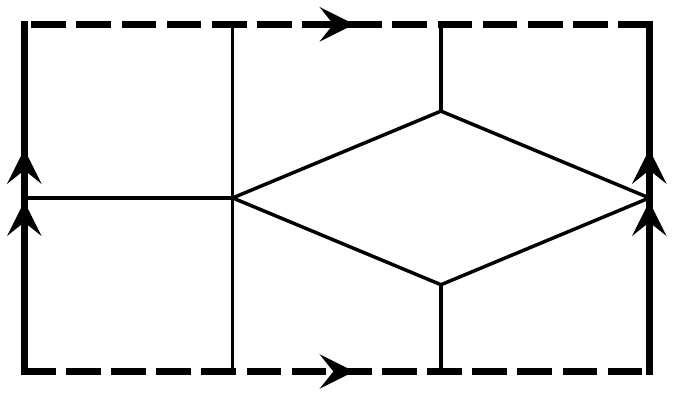}}
\caption[Quadrangulations of the torus]
{(a)~A $2,\!6$-quadrangulation of the torus with two vertices,
and its refinement with eight vertices.
(b)~A $3^25^2$-quadrangulation of the torus with seven vertices.
Erasing the dashed edges would give
a smaller example containing only the four exceptional vertices.
}
\label{fig:quadrs}
\end{figure}
Finally, one can consider hexangulations of the torus, with average
vertex degree~$3$. In this case, the corresponding result takes a
different form. Hexangulations with two irregular vertices of
degree~$2$ and~$4$ exist, as shown in Figure~\ref{fig:hexes}.
\begin{figure}
  \centering%
  (a)
  \raisebox{-40pt}{\includegraphics[scale=.6]{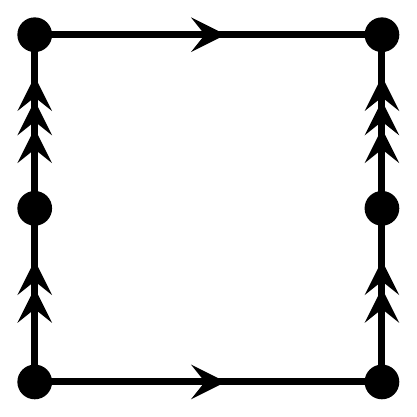}}\hspace{1em}
  (b)
  \raisebox{-40pt}{\includegraphics[scale=.6]{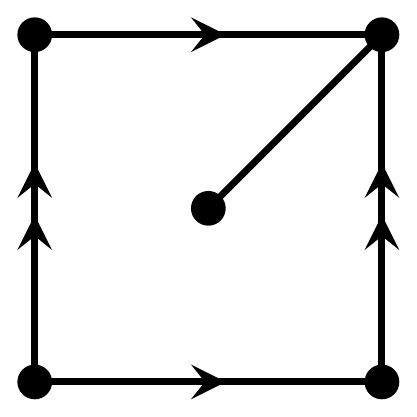}}\hspace{1em}
  (c)
  \raisebox{-40pt}{\includegraphics[scale=.6]{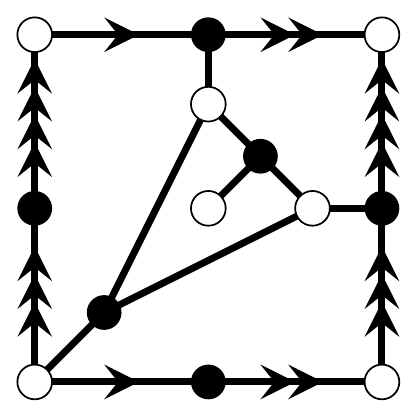}}\hspace{1em}
  (d)\!
  \raisebox{-40pt}{\includegraphics[scale=.35]{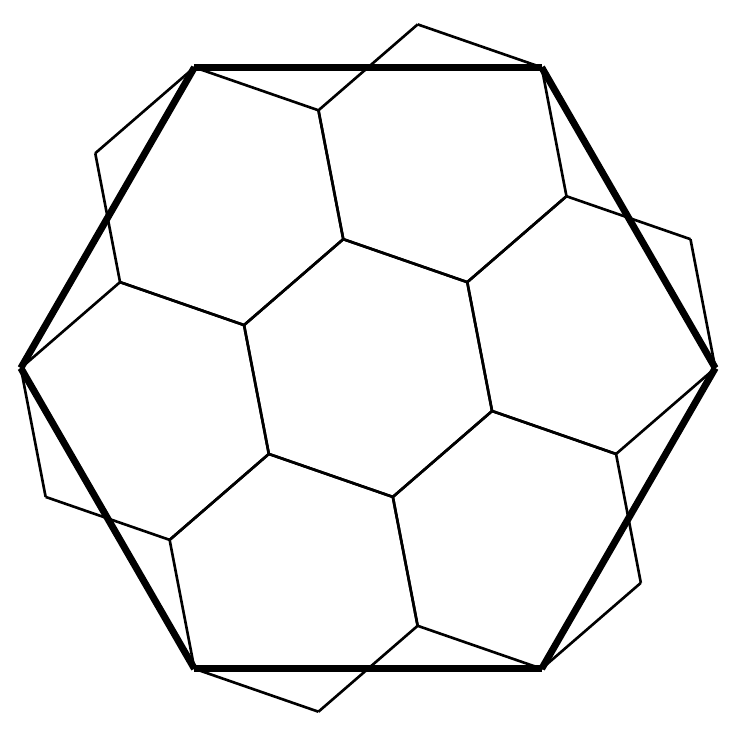}}
\caption[Hexangulations of the torus]
{(a)~A $2,\!4$-hexangulation of the torus. 
The $1$-skeleton fails to be bipartite because it has a loop edge.
(b)~A $1,\!5$-hexangulation.
(c)~A $1,\!5$-hexangulation with bipartite $1$-skeleton.
(d)~A subdivision scheme for hexangulations, which can be used to
generate bigger examples. 
}\label{fig:hexes}
\end{figure}
But any such hexangulation has an odd
edge-cycle:

\begin{thm}
  \label{thm:24}
  The vertices of a $2,\!4$-hexangulation of the torus cannot be
  bicolored (that is, the $1$-skeleton is not bipartite).
\end{thm}

Similarly, $4,\!8$-triangulations and $2,\!6$-quadrangulations are subject
to combinatorial restrictions:

\begin{thm}
  \label{thm:48triang}
  The faces of a $4,\!8$-triangulation of the torus cannot be
  $2$-colored (that is, the dual graph is not bipartite).
\end{thm}

\begin{thm}
  \label{thm:26quad}
  The edges of a $2,\!6$-quadrangulation of the torus cannot be
  bicolored with colors alternating around each face
(or equivalently, around each vertex).
\end{thm}

We prove Theorems \ref{thm:57}--\ref{thm:26quad} by converting the
statements about combinatorics to statements about geometry
(Sections~\ref{sec:cone_metrics} and~\ref{sec:holonomy}).
Namely, any triangulated torus has a natural \emph{equilateral} metric
obtained by declaring each edge to have length~$1$ and each triangle
to be a euclidean equilateral triangle.  This approach was also used
by Thurston~\cite{Thu98} to classify triangulations of the sphere with
vertex degrees at most~$6$.  In our case, if there were a
$5,\!7$-triangulation of the torus, its equilateral metric would be a
euclidean metric with exactly two cone singularities
(Section~\ref{sec:cone_metrics}).  We obtain a contradiction by
studying the possible holonomy groups of such metrics
(Section~\ref{sec:holonomy}).
Specifically, we prove the following theorem.

\begin{holthm}
  Suppose the torus is equipped with a euclidean cone metric with
  exactly two cone points $p_{\pm}$ of curvature $\pm 2\pi/n$, for
  some integer $n\geq 2$. Then the holonomy group $H$ contains the cyclic
  group $C_{n}$ of order $n$ as a proper subgroup: $C_{n}\lneqq H$.
\end{holthm}

We provide two proofs for the Holonomy Theorem. One argument is
metric in nature (Section~\ref{sec:metric_proof}), the other relies
on the induced conformal structure and proceeds by invoking the
residue theorem (Section~\ref{sec:conformal_proof}).

\foreignlanguage{slovak}{Jendro^l \& Jucovi^c}~\cite{JeJu} actually
prove a stronger statement than Theorem~\ref{thm:57}. They also show
that every distribution of irregular degrees except $5,7$ does indeed
occur in some triangulation as long as the average is $6$. (Their
notion of triangulation is more restrictive. See
Section~\ref{sec:maps} for the precise statements.) Barnette,
\foreignlanguage{slovak}{Jucovi^c} \& Trenkler~\cite{BJT} prove an
analogous existence result as well. These existence proofs involve
more or less explicit constructions and are rather different in nature
from the non-existence proofs. In this article, we do not deal with
such existence statements at all. Neither are we concerned with the
realization of tori as polyhedral surfaces in $\R^{3}$ (as considered,
for instance, in \cite{GrSz} and the references therein).

We have phrased 
Theorems~\ref{thm:57}--\ref{thm:24}
in terms of triangulations, quadrangulations
and hexangulations with some exceptional vertices.  One could reformulate
these results in dual terms:  For example, dual to a triangulation is
a map where the vertices all have degree three.  Hexagon faces would
then be considered regular, while pentagons and heptagons would be exceptional.
In fact, Theorem~\ref{thm:57} was
originally stated in that dual form. We prefer our formulation
because it is more closely connected to the euclidean cone metric we use. 

\section{Maps on surfaces}

\label{sec:maps}

Suppose $M$ is a closed 
connected
surface.  For us, a \emph{map} on~$M$ will
mean an embedding of a finite graph~$G$
in~$M$ such that each face is topologically an open disk.  Here, a \emph{face} 
is a component of the complement $M\setm G$, 
and a \emph{graph} is a one-dimensional cell complex. In particular, a
graph can have loops and multiple edges.
This notion of map corresponds to the construction
of surfaces by starting with a collection of
(topological) polygons (including $1$- and $2$-gons)
and gluing their edges in pairs.  Any map has a combinatorial dual,
with a vertex for every face and vice versa.
A map on $M$ is called a \emph{triangulation} (\emph{quadrangulation},
\emph{hexangulation}) of $M$ if all faces are triangles
(quadrilaterals, hexagons).

Some authors put more restrictions on the combinatorics of a map.
Suppose (a)~the 
graph~$G$ has no loops or multiple edges
and (b)~the boundary of each face is an embedded circle
and (c)~any two faces meet either along a single edge or
at a single vertex or not at all.  Then we will call the map
\emph{polyhedral}.
In a polyhedral map,
each vertex has degree $d\ge 3$ and each face has $k\ge 3$ sides.
A polyhedral triangulation is a simplicial complex.  For a map on the
sphere, the following statements are equivalent~\cite[Section~13.1]{Gru}:
(i)~the map is polyhedral; (ii)~the map can be realized by a
convex polyhedron; (iii)~the graph is $3$-connected;
(iv)~the dual map is polyhedral.

Consider a map with $V$ vertices, $E$ edges, and $F$ faces on a
surface with Euler characteristic~$\chi$. Let $v_{k}$ be the
number of vertices with degree $k$ and let $p_{k}$ be the number of
$k$-gons among the faces. The Euler equation $\chi = V-E+F$ (where
of course
$V=\sum_{k} v_k$ and $F=\sum_k p_k$)
and the double counting formula
$\sum_{k}kv_{k}=\sum_{k}kp_{k}=2E$ imply relations between the
numbers $v_{k}$ and $p_{k}$. If all faces are $n$-gons, one obtains
\begin{equation}
  \label{eq:degree_sequence_n-gons}
  \sum_{k}(\bar{n}-k)v_{k}=\bar{n}\chi,
\end{equation}
where $\bar{n}:=2n/(n-2)$. In the case $\chi=0$, we see 
the average vertex degree is exactly~$\bar{n}$.
(For any value of~$\chi$, the average vertex degree for large maps
with $n$-gonal faces approaches~$\bar n$.)

Note that $\bar n$ is an integer for $n=3,4,6$; then
$v_{\bar n}$ does not appear in equation~\eqref{eq:degree_sequence_n-gons}.
That is, the Euler relation imposes no restictions on $v_{\bar n}$
in these cases.
For polyhedral
triangulations ($n=3$, $\bar{n}=6$) of the sphere,
equation~\eqref{eq:degree_sequence_n-gons} is the only
relation among the $v_{k}$ for $k\not=6$:

\begin{thm}[Eberhard~\cite{Ebe}]  
  Suppose $(v_{3},v_{4},v_{5};v_{7},v_{8},\ldots)$ is a sequence of
  nonnegative integers with finitely many nonzero terms satisfying
  $\sum(6-k)v_{k}=12$. Then for some $v_6\ge0$ there exists a
  polyhedral triangulation of the
  sphere with $v_{k}$ vertices of degree $k$ for all $k$.
\end{thm}

The theorem is usually phrased in dual terms as a theorem about polyhedra with
vertices of degree $3$ and prescribed number of $k$-gons for $k\not=6$;
a proof can be found in \cite[Section~13.3]{Gru}.
In some cases, the possible values of~$v_6$ are known exactly.
(See \cite[Section~13.4]{Gru} and also \cite[Section~6]{Thu98}.)
For instance, a $5^{12}$-triangulation of the sphere exists
exactly for $v_6\ne 1$; a $3^4$-triangulation exists exactly
for $v_6$ even and not equal to~$2$.

\foreignlanguage{slovak}{Jendro^l \& Jucovi^c} proved an
analogous theorem for the torus and found that one exceptional case
has to be excluded:

\begin{thm}[\foreignlanguage{slovak}{Jendro^l \& Jucovi^c}~\cite{JeJu}]
  Suppose $p=(p_{3},p_{4},p_{5};p_{7},p_{8},\ldots)$ is a sequence of
  nonnegative integers with finitely many nonzero terms satisfying
  $\sum(6-k)p_{k}=0$. Then there
  exists a nonnegative integer $p_{6}$ and a map on the torus with
  $3$-connected $3$-valent graph having $p_{k}$ $k$-gons
  if and only if $p\not=(0,0,1;1,0,\ldots).$
\end{thm}

We formulated the non-existence statement for the exceptional case in
dual form as Theorem~\ref{thm:57}.
\foreignlanguage{slovak}{Jendro^l \& Jucovi^c} only
consider $3$-connected graphs, but they do not
seem to use this assumption in their non-existence proof. In an
earlier paper, Gr{\"u}nbaum had not noticed the exceptional case and
made a wrong claim~\cite[Theorem 3]{Gru68}.

Without the assumption that all faces have the same number of
sides, one obtains the following equation, which is symmetric in $p$ and
$v$:
\begin{equation}
\label{eq:degree_sequence_general}
  \sum_{k}(4-k)p_{k}+\sum_{k}(4-k)v_{k}=4\chi.
\end{equation}
Note that $p_{4}$ and $v_{4}$ do not occur. The double counting
formula implies
$\sum_{k\not=4} kp_k = 2E - 4p_4$ and $\sum_{k\not=4} kv_k = 2E - 4v_4$, so
that
\begin{equation}
  \label{eq:evenness_condition}
  \sum_{k\not=4}kp_{k}\quad\text{and}\quad\sum_{k\not=4}kv_{k}\quad\text{are even.}
\end{equation}

Then it is natural to pose the following question:
  Given $p_{k}$ and $v_{k}$ for $k\not=4$ satisfying
  \eqref{eq:degree_sequence_general} and
  \eqref{eq:evenness_condition}, is there a corresponding map?
For the case of the sphere (and considering only polyhedral maps),
Gr{\"u}nbaum~\cite{Gru69} showed that the answer is yes. In the case
of the torus, only two dual exceptional cases must be excluded:

\begin{thm}[Barnette, \foreignlanguage{slovak}{Jucovi^c} \& Trenkler~\cite{BJT}]
  Suppose two sequences $p=(p_{3};p_{5},p_{6},p_{7}\ldots)$ and
  $v=(v_{3};v_{5},v_{6},v_{7}\ldots)$ of nonnegative
  integers with finitely many nonzero terms are given, satisfying
  \eqref{eq:degree_sequence_general} and~\eqref{eq:evenness_condition}.
  Then there exist $p_4,v_4\ge0$
  and a map on the torus with $3$-connected graph
  having $v_{k}$ vertices of degree~$k$ and $p_{k}$ $k$-gons for each
  $k\geq 3$, if and only if it is not the case that
  $p=(1;1,0,0\ldots)$ and $v=(0;0,\ldots)$ or dually that
  $p=(0;0,\ldots)$ and $v=(1;1,0,0\ldots)$.
\end{thm}

We formulated the non-existence statement
as Theorem~\ref{thm:35}; the assumption of
$3$-connectedness is not necessary.
(Another proof of the existence
statement in the case $v=(0;0,\ldots)$ is due to Zaks~\cite{Zak}.)

\foreignlanguage{slovak}{Jucovi^c} \& Trenkler~\cite{JuTr} also
answered the above question for closed orientable surfaces of genus
$g\geq 2$ (assuming that all faces and vertices have degree at least
$3$). The answer is yes, with no exceptional cases.

\section{Euclidean cone metrics}
\label{sec:cone_metrics}
Let $\omega > 0$ be different from $2\pi$.
A \emph{euclidean cone} of angle~$\omega$ is the metric space
resulting from
gluing together the two edges of a planar wedge of angle~$\omega$
(when $\omega > 2\pi$, we are gluing together several wedges of total angle $\omega$).
A \emph{euclidean cone metric} on a surface~$M$ is one in which
every point has a neighborhood isometric either to an open subset
of the euclidean plane or to a neighborhood of the apex of a euclidean
cone. Points mapped to an apex are called \emph{cone points} and
form a discrete subset of $M$. Each cone point has a \emph{curvature}
$\kappa := 2\pi - \omega$, where $\omega$ is the angle of the
corresponding euclidean cone. The complement of the set of cone
points is denoted by $M^o$ and is clearly locally euclidean.

The metric induced on any polyhedral surface in~$\R^n$
is such a euclidean cone metric.  Indeed, points
along the edges and in the interior of faces have flat neighborhoods,
but a vertex~$v$ is a cone point (unless the angle sum around~$v$
equals $2\pi$).

Suppose we are given a map on a surface~$M$
without $1$- or $2$-sided faces.
This map induces a euclidean cone
metric on~$M$, called the \emph{equilateral metric},
as follows: each edge is a segment of length~$1$,
and each $k$-sided face is isometric to a euclidean regular $k$-gon.
If all faces are $n$-gons, a vertex of degree $k$ has curvature
$2\pi(1 - k/\bar{n})$, where $\bar{n}=2n/(n-2)$ as in
equation~\eqref{eq:degree_sequence_n-gons}.
Thus equation~\eqref{eq:degree_sequence_n-gons}
is exactly the Gauss--Bonnet theorem for the equilateral metric.

In particular, given a triangulation, quadrangulation or hexangulation of a torus, it is only the \emph{exceptional} vertices (those not of degree $\bar{n}$) that become cone points in the equilateral metric.
This observation allows us to give an easy
proof of the classification of (degree-)regular tilings of the torus.  This has
been treated by many authors \cite{Alt73,Neg83,Thomassen,DU06,BrKu},
although most of the results are restricted to the class of polyhedral maps.

\begin{thm}
Any triangulation, quadrangulation or hexangulation of the
torus with no exceptional vertices is a quotient of the
corresponding infinite regular tiling of the plane.  Equivalently,
it is a finite cover of the $1$-vertex triangulation
(Figure~\ref{fig:reg-examples}(a)), the $1$-vertex quadrangulation,
or the $2$-vertex hexangulation (visible in Figure~\ref{fig:reg-examples}(c)),
respectively.
\end{thm}
\begin{proof}
Since there are no exceptional vertices, the equilateral
metric is a euclidean (flat) metric on the torus~$M$.  Thus
its universal cover is the euclidean plane~$\R^2$, so 
$M$ is the 
quotient by some lattice~$\Lambda$
of translations.  The map on~$M$ pulls back to the cover,
giving the infinite tiling~$T$ by regular $k$-gons.
The translational symmetries of~$T$ form a lattice $\Lambda_0$,
and $T/\Lambda_0$ is the corresponding minimal map on the torus.
Since~$\Lambda$ preserves~$T$, we have $\Lambda<\Lambda_0$.
The covering of $\R^2/\Lambda_0$ by $M=\R^2/\Lambda$ restricts
to a covering of the minimal map $T/\Lambda_0$ by the original map
$T/\Lambda$.
\qed\end{proof}

\section{Holonomy groups and the proofs of Theorems~\ref{thm:57}--\ref{thm:26quad}}
\label{sec:holonomy}

Given a euclidean cone metric on an oriented surface~$M$,
the holonomy group~$H(M)$ is defined as follows.  Fix
a basepoint $x\in M^o$, and consider a loop~$\gamma$
in~$M^o$ based at~$x$.  Parallel transport along~$\gamma$
induces a rotation $h(\gamma)$ of the tangent space~$T_xM$, that is,
an element of $SO_2$ called the holonomy of~$\gamma$.
Because the metric on~$M^o$ is flat,
this holonomy is unchanged if we replace~$\gamma$ by
a homotopic loop.  Thus we get a map $\pi_1(M^o,x)\to SO_2$,
which is independent of the choice of basepoint~$x$ since $SO_2$ is abelian.
Its image is the \emph{holonomy group} $H(M)<SO_2$ of~$M$.

Note that $\pi_1(M^o,x)$ is generated by (generators of) $\pi_1(M,x)$
together with loops around each of the cone points.
A loop around a single cone point of curvature~$\kappa$ has
holonomy~$e^{i\kappa}$. 
We will be particularly interested in the case where the
holonomy group is finite.  The finite subgroups of $SO_2$
are of course the cyclic groups 
$C_n := \langle e^{2\pi i/n}\rangle$.

\begin{lem}\label{lem:holonomy}
Suppose we are given a triangulation, quadrangulation
or hexangulation of a surface~$M$. 
Then the holonomy group of the associated
equilateral metric is a subgroup of~$C_6$ (for triangulations
or hexangulations) or a subgroup of~$C_4$ (for quadrangulations).
\end{lem}

\begin{proof}
Take a basepoint $x \in M^o$ lying on an edge of the map.
Consider parallel translation of a vector~$v$ along a loop~$\gamma$,
which we may assume is transverse to the edges of the map.
Each time $\gamma$ crosses an edge, look at the angle between~$v$
and a vector along that edge. Successive angles will differ
by the angle between two edges of some face of the map.
Since that face is a euclidean regular $k$-gon, the angles differ by
a multiple of $\pi/3$ in the case of triangulations and hexangulations,
and by a multiple of $\pi/2$ in the case of quadrangulations.
Summing these changes, the same is true for the angle between~$v$
and its parallel translate along $\gamma$.
\qed\end{proof}

Stronger statements about the holonomy can be made if the map has
additional combinatorial properties.

\begin{lem}\label{lem:small-hol}
For a hexangulation, the holonomy group is a subgroup of~$C_3$
if and only if the vertices can be $2$-colored
(that is, the $1$-skeleton is bipartite).
For a triangulation, $H<C_3$
if and only if the triangles can be $2$-colored
(that is, the dual graph is bipartite).
For a quadrangulation, $H<C_2$
if and only if the edges can be $2$-colored
(with colors alternating around each face
or, equivalently, around each vertex).
\end{lem}

\begin{proof}
In a hexangulation with $2$-colored vertices, orient
each edge from black to white; in a triangulation with $2$-colored
faces, orient each edge with black to its left.
Now repeat the parallel translation argument of Lemma~\ref{lem:holonomy}
but keeping track of the angle between $v$ and the \emph{oriented}
edges.  Since this changes only by multiples of $2\pi/3$, we get $H<C_3$.

For a quadrangulation with $2$-colored edges,
we claim that the angles from $v$ to edges of the same color are
congruent modulo $\pi$, while the angles from $v$ to edges of different
colors differ by odd multiples of $\pi/2$. This is easily proved
by looking at the moments when $\gamma$ enters and leaves a face.
Since $\gamma$ starts and ends on the same edge, $v$ comes back either
unchanged or rotated by $\pi$.

In all three cases, the converse follows by propagating a $2$-coloring along
arbitrary paths; the holonomy condition shows we will never encounter
a contradiction.
\qed\end{proof}

Our proofs Lemmas~\ref{lem:holonomy} and~\ref{lem:small-hol}
could alternatively be phrased in terms of the developing map
from the universal cover of~$M^o$ to the euclidean plane, whose
image is the corresponding regular tiling.
Note that the regular tilings can be colored in the ways described
and the rotational symmetries of the colored tilings are $C_3$ or $C_2$.

Recall that the Holonomy Theorem, our main tool,
restricts the holonomy group of a torus with just two cone points.
Before proving it in Sections~\ref{sec:metric_proof}
and~\ref{sec:conformal_proof} below,
we apply it to prove the results we listed as
Theorems \ref{thm:57}--\ref{thm:26quad}.

\begin{proof}[of Theorems \ref{thm:57}--\ref{thm:26quad}]
In the situations of Theorems \ref{thm:57}--\ref{thm:26quad},
the holonomy group is a subgroup of $C_n$ with $n := 6,4,3,3,2$, respectively,
according to Lemmas~\ref{lem:holonomy} and~\ref{lem:small-hol}.
On the other hand, the assumptions on the exceptional vertices imply
that the equilateral metric is a euclidean cone metric with two cone points
of curvature~$\pm 2\pi/n$. Then, by the Holonomy Theorem,
$C_n$ is a proper subgroup of the holonomy group.
This contradiction proves the theorems.
\qed\end{proof}

Because there are no intermediate subgroups between~$C_3$ or~$C_2$
and~$C_6$---or between~$C_2$ and~$C_4$---Lemma~\ref{lem:holonomy}
and the Holonomy Theorem immediately give the following:

\begin{cor}\label{cor:48holo}
The equilateral metric on any $4,\!8$- or $3,\!9$-triangulation
of the torus has holonomy exactly $H=C_6$.
Similarly, any $2,\!6$-quadrangulation has $H=C_4$ and any
$2,\!4$-hexangulation has $H=C_6$.
\end{cor}

\section{Proof of the Holonomy Theorem 
using metric geometry}
\label{sec:metric_proof}

Our first proof of the Holonomy Theorem starts with a
torus with two cone points of curvature $\pm2\pi/n$.
By examining the structure of closed geodesics on this
surface, we explicitly find a loop whose holonomy
is a rotation by a smaller angle.

\begin{proof}[of the Holonomy Theorem]
We are given a torus with exactly two cone points $p_\pm$,
of curvature $\pm2\pi/n$.
Since the holonomy around 
$p_{\pm}$ is $e^{\pm 2\pi i/n}$, we know $C_n$ is a subgroup of $H$
and we must prove they are not equal.

Let $\gamma$ be any shortest non-contractible loop on~$M$.
Away from the cone points, it must be a geodesic in the usual
sense.  If it passes through a cone point~$p$
with angle~$\omega$, the angles to the left and right
of~$\gamma$ at~$p$ sum to~$\omega$.  But each of these angles
must be at least~$\pi$, for otherwise we could shorten~$\gamma$
by moving it off~$p$ to that side.  Thus we see that~$\gamma$
cannot pass through the cone point of positive curvature.

On the other hand, we are free to assume that $\gamma$ does
pass through~$p_-$.  If not, then near any point it looks
like a straight segment, so $\gamma$ has a neighborhood isometric
to a euclidean cylinder.  Thus $\gamma$ can be translated sideways,
while remaining a shortest geodesic.
This translation can be continued until we hit a singularity,
which must be~$p_-$.

So let $\pi+\alpha$ and $\pi+\beta$ be the angles formed by~$\gamma$ at~$p_-$,
with $\alpha\ge\beta\ge0$ and $\alpha+\beta=2\pi/n$.
If we consider curves parallel to~$\gamma$
just to either side, they have holonomy~$\alpha$ and~$\beta$.
Thus unless $\beta=0$, we have $H\ne C_n$, as desired.

If $\beta=0$ then $\gamma$ has to one side a euclidean cylinder neighborhood,
and as above, it can be translated sideways through this neighborhood.
Again, it will never hit a cone point of positive curvature.  This time
we stop when the translated~$\gamma'$ first touches~$\gamma$.  This
first contact must happen (only) at~$p_-$, for if it happened at
a regular point, $\gamma$ and~$\gamma'$ would coincide, and we would
have traced out the whole torus without seeing the positive cone point.

\begin{figure}
  \centering%
  \includegraphics[width=0.7\textwidth]{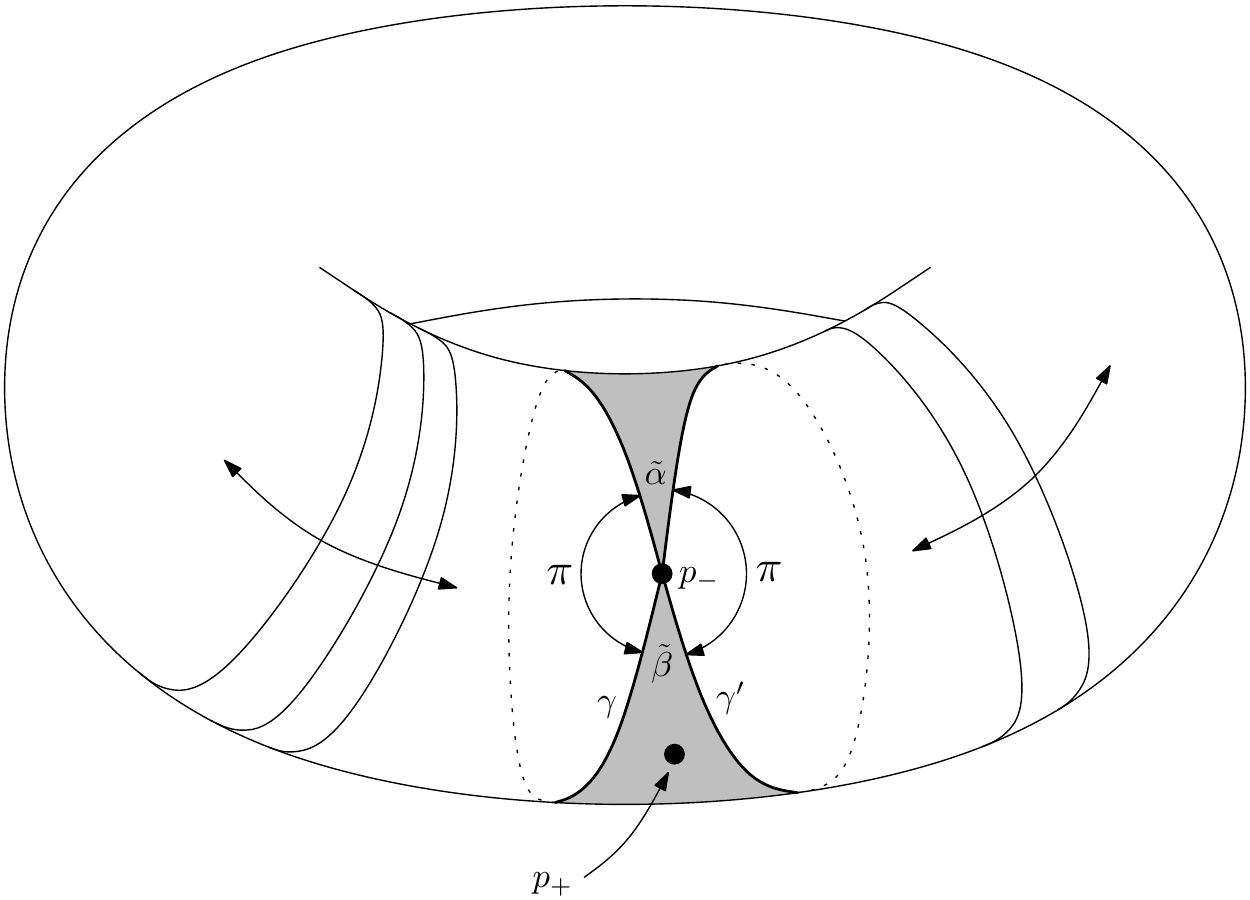}
  \caption{A schematic drawing of the parallel translates of the geodesic
     $\gamma$ on the torus.  The point $p_{+}$ lies in the shaded digon
     with angles 
     $\tilde\alpha$ and $\tilde\beta$;
     the unshaded region is a
     euclidean cylinder.} \label{fig:torus}
\end{figure}

It follows that $\gamma$ and~$\gamma'$ bound a digon (both of whose
vertices are at~$p_-$) containing the point $p_{+}$,
as shown in Figure~\ref{fig:torus}.
Let 
$\tilde\alpha\ge\tilde\beta>0$ 
be the angles of the digon, with 
$\tilde\alpha+\tilde\beta=2\pi/n$.
The complement of the digon is a euclidean cylinder, also meeting itself
at~$p_-$, foliated by translates of~$\gamma$.
Consider a closed path based at~$p_-$ running once along the length
of this cylinder.  If it is perturbed off~$p_-$ into the 
angle~$\tilde\alpha$
then its holonomy is 
$\tilde\alpha\in(0,2\pi/n)$.
This proves, as desired, that~$H$ is bigger than~$C_n$.
\qed\end{proof}

\begin{rmk}
There are many different equivalent ways to phrase this argument.
For instance, given a $5,\!7$-triangulation, one could cut open
the torus along geodesic arcs connecting the cone points to
give a planar polygon.  Its vertices would lie in the triangular
lattice in the plane.  Our statements about curvature and holonomy
would become conditions on the angles of such a polygon.
\end{rmk}

Figure~\ref{fig:two_sings} shows a simple example of a euclidean cone metric on the torus with two cone points of arbitrary curvatures $\pm\kappa$.
\begin{figure}
  \centering
  \begin{overpic}[scale=.6]{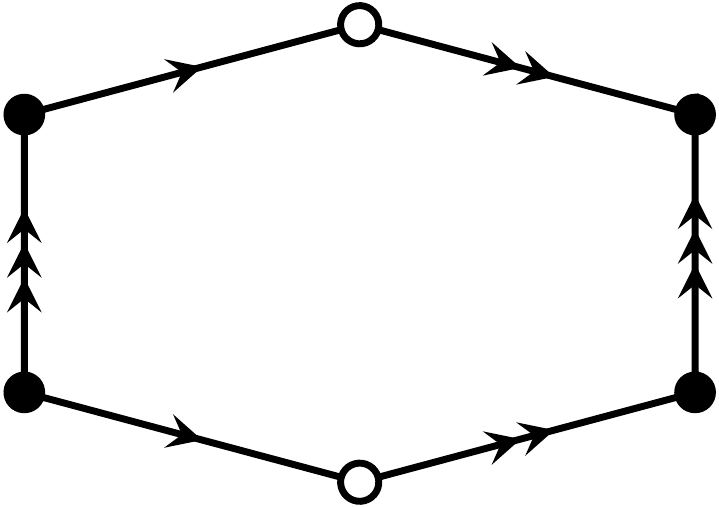}
    \put(7,21){$\frac{\pi}{2}+\frac{\kappa}{4}$}
    \put(40,53){$\pi-\frac{\kappa}{2}$}
  \end{overpic}
  \caption{Example of a torus equipped with a euclidean cone metric
    with two cone singularities of curvature $\pm\kappa$.
    The holonomy group is generated by a rotation through angle $\kappa/2$.}
  \label{fig:two_sings}
\end{figure}

\section{Proof of the Holonomy Theorem 
using conformal geometry}
\label{sec:conformal_proof}

Our second proof of the Holonomy Theorem starts instead with a torus~$M$
with a euclidean cone metric whose holonomy group is known to be
$C_{n}=\langle e^{2\pi i/n}\rangle$.
Each cone point $p_j$ has curvature $\kappa_j =: 2\pi k_j/n$,
an integral multiple of $2\pi/n$.
The formal linear combination $\sum k_j p_j$ is then a divisor on~$M$.
The lemma below shows that it is a \emph{principal divisor} for the conformal
structure induced by the cone metric on~$M$. That is, there
exists an \emph{elliptic function} (a meromorphic function on the torus,
or equivalently a doubly periodic function on its universal cover~$\C$)
with a zero of order $k_j$ at each cone point $p_{j}$ of positive curvature,
a pole of order $-k_j$ at each cone point $p_{j}$ of negative curvature,
and no other zeros or poles.

\begin{proof}[of the Holonomy Theorem]
Cauchy's residue theorem, applied to a fundamental domain, shows
that the residues of any elliptic function at its poles sum to zero.
In particular, there is no elliptic function
with a single simple pole~\cite[Theorem~4, p.~271]{Ahl},
meaning that no divisor of the form $p_+-p_-$ is principal.
But given a torus with holonomy $H=C_n$ and cone points $p_j$ of
curvature $2\pi k_j/n$, Lemma~\ref{lem:principal_divisor}
says $\sum k_j p_j$ is a principal divisor.  Thus if there are just
two cone points, their curvatures cannot be $\pm 2\pi/n$ but instead
must be some larger multiple of this.
\qed\end{proof}

\begin{lem}
  \label{lem:principal_divisor}
  Suppose the torus $M$ has a euclidean cone metric with holonomy group
  $H=C_{n}$ and cone points $p_{j}$ with curvature $2\pi k_{j} /n$.
  With respect to the conformal structure on~$M$ induced by the cone metric,
  the divisor $\sum k_{j}p_{j}$ is principal.
\end{lem}

\begin{proof}
The universal cover of $M^{o}=M\setm\{p_{1},\ldots,p_{m}\}$ extends, via
metric completion, to a branched cover $\hat M\rightarrow M$,
ramified over the singularities~$p_{j}$. Let $w:\hat M\rightarrow \C$
be the developing map (where $\C$ is equipped with the standard
euclidean metric). We will view the function~$w$ also as a branched
multivalued function on $M$, and our proof will proceed by analyzing
$(dw)^n$ as a meromorphic differential of degree~$n$ on~$M$.

Recall that the conformal structure of $M$ is defined by the following
atlas. In a sufficiently small neighborhood of a nonsingular point,
any branch of $w$ may serve as a coordinate. In a sufficiently small
neighborhood of a singular point $p_{j}$, choose any connected set of
branches of $w$ and let $w_{j}\in\C$ be their common value at
$p_{j}$. Since the cone angle at $p_{j}$ is $2\pi(n-k_{j})/n$, a
coordinate function around $p_{j}$ can be defined by
\begin{equation}
\label{eq:u_coord}
u_{j}=(w-w_{j})^{n/(n-k_{j})}.
\end{equation}
(The expression on the right hand side is multivalued on $M$, but
it is unramified at $p_{j}$; any choice of branch will do.)

Note that for any deck transformation
$\phi:\hat M\rightarrow\hat M$, there exist $a\in H$ and $b\in\C$ such that
\begin{equation*}
  w\circ\phi = aw+b. 
\end{equation*}
This implies $\phi^{*}\,dw=a\,dw$. By assumption, $H$ is
generated by $e^{2\pi i/n}$, so $a^{n}=1$ and $(dw)^{n}$ is a
well-defined meromorphic differential of degree~$n$ on~$M$,
that is, a meromorphic section of $K^{n}$, where $K$ is the canonical bundle.

Near a cone point $p_{j}$, equation~\eqref{eq:u_coord} implies
\begin{equation*}
  dw = \tfrac{n-k_j}{n}\,u_{j}^{-k_{j}/n}\,du_{j},
\end{equation*}
giving
\begin{equation*}
  (dw)^{n}=\big(\tfrac{n-k_j}{n}\big)^{n}\,u_{j}^{-k_{j}}\,(du_{j})^{n}.
\end{equation*}
Thus we see that $(dw)^n$ has poles of order~$k_{j}$ at the
cone points~$p_{j}$ of positive curvature, and zeros of
order~$-k_{j}$ at the cone points~$p_{j}$ of negative curvature.
That is, its divisor is $\sum -k_j p_j$.

On the torus $M=\C/\Lambda$ (or more properly on its universal cover)
there is a global uniformizing coordinate~$z$.  Then $dz$ is a nonvanishing
holomorphic differential on~$M$, a section of~$K$.  It follows that
$(dz)^{n}/(dw)^{n}$ is a meromorphic function with divisor $\sum k_{j}p_{j}$,
so this divisor is principal.
\qed\end{proof}

\section{Burgers vectors and a theorem of Dress}
\label{sec:Burgers}

Since we have shown there is no there is no $5,\!7$-triangulation
of the torus, it may be surprising that
it is possible to find a $5,\!7$-triangulation of the infinite plane.
Figure~\ref{fig:KS57ex} shows an example nicely laid out in the
euclidean plane by Ken Stephenson with his \emph{CirclePack} program.
\begin{figure}
  \centering%
  \includegraphics[scale=.7, trim=30pt 50pt 0 30pt, clip]{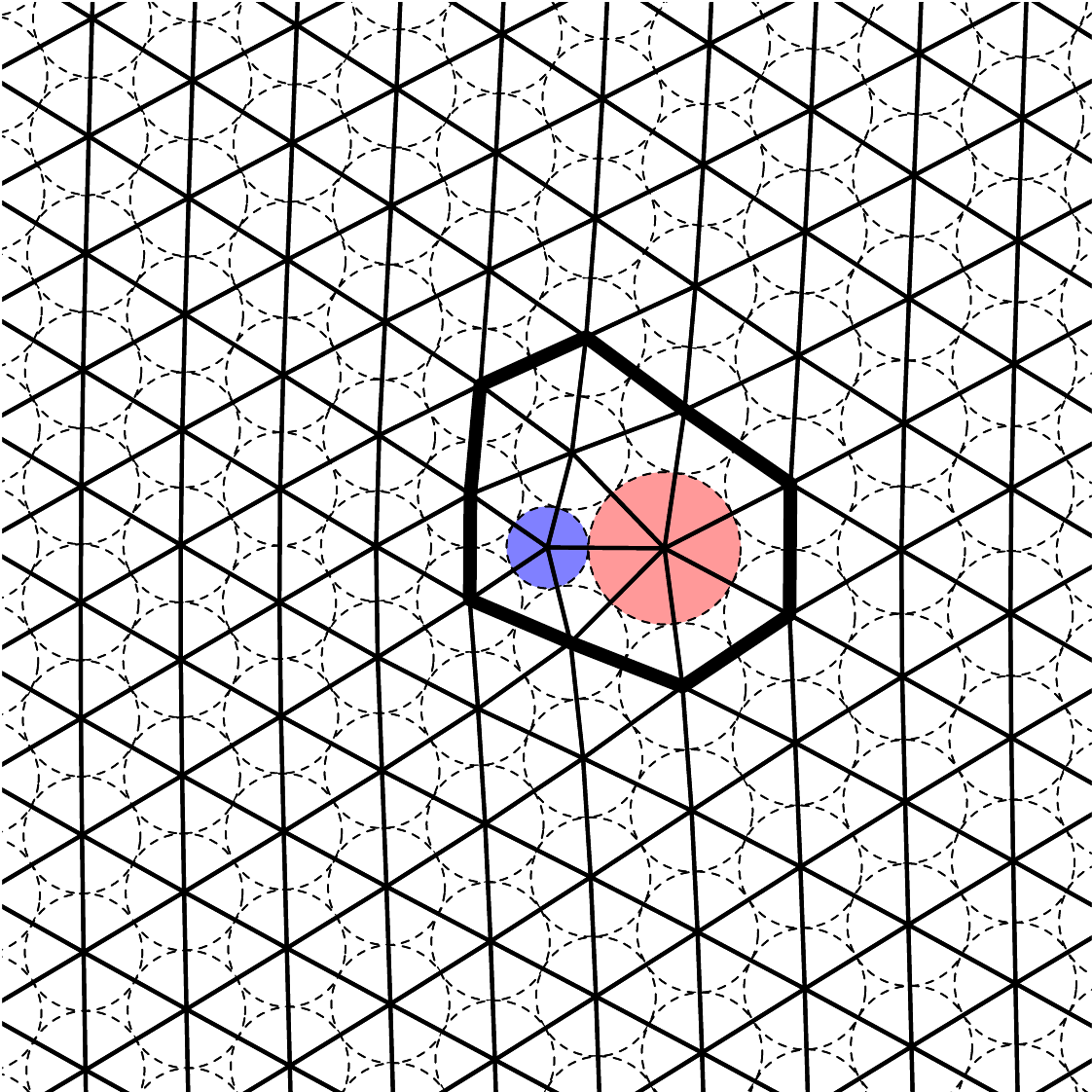}%
  \hspace{1.5cm}%
  \raisebox{30pt}{\includegraphics[scale=.3]{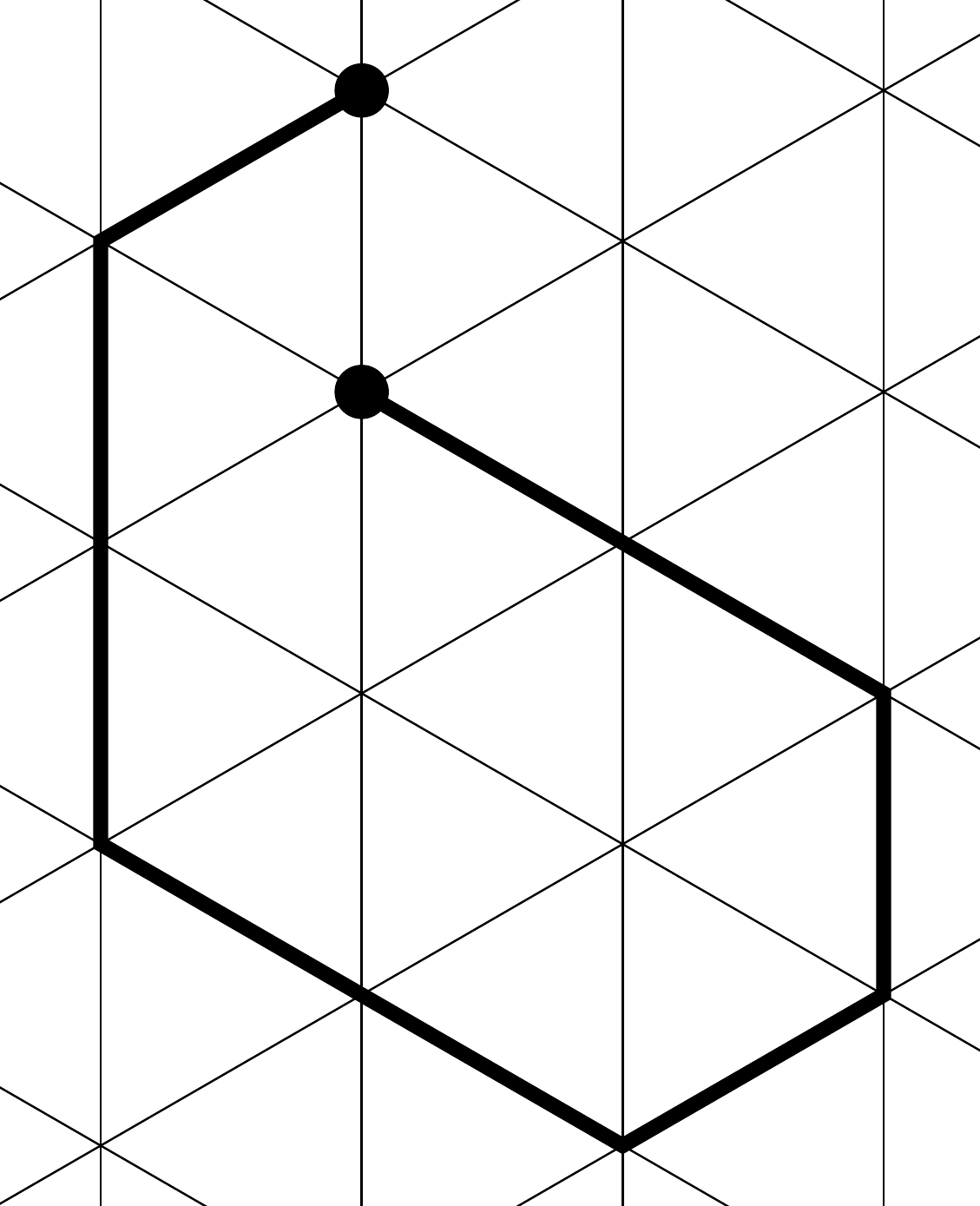}}%
\caption[A 57--triangulation of the plane]
{This figure (left), computed by Ken Stephenson with \emph{CirclePack},
shows a $5,\!7$-triangulation of the plane, in which the two
exceptional vertices are adjacent, forming a dislocation in the
hexagonal lattice.  The heavy lines show a loop around the dislocation,
from which the Burgers vector can be calculated: A path in the regular
hexagonal lattice (right) with the same $9$~steps and $6$~left turns would fail
to close; the Burgers vector is the resulting gap, here one vertical step.
The nonzero Burgers vector shows this triangulation is not isomorphic
to the regular one near infinity---no matter how far out we go, we
still measure the same Burgers vector for this dislocation.} \label{fig:KS57ex}
\end{figure}

We note, however, that this triangulation is not isomorphic
near infinity to the regular one.  Physicists and crystallographers
measure the difference---a dislocation in the lattice---by the
so-called Burgers vector.  As shown
in Figure~\ref{fig:KS57ex}, if a closed path enclosing both exceptional
vertices is transferred onto the regular triangular lattice,
it will fail to close.  The gap is the Burgers vector,
and is independent of the path chosen; the dislocation can be
measured near its source or arbitrarily far away.

Andreas Dress used similar ideas in~\cite{Dre} to study triangulations of
the plane that are isomorphic to the regular one outside some bounded region.
Such a triangulation has a finite number of exceptional vertices.
Dress sketched a proof that the number of exceptional vertices cannot
be one or two.  The first step is to consider a large rectangle
enclosing the exceptional vertices, whose boundary is within the regular
part of the triangulation.  Its opposite sides can thus be glued to
form a torus.
The Euler characteristic then shows immediately that there
cannot be a single exceptional vertex.  The importance
of Dress's theorem is that it also rules out the case of two cone points.
Thus, for instance, a $5,\!7$-triangulation of the plane cannot
be isomorphic to the regular triangulation near infinity.
Of course, applying edge flips to the regular triangulation
produces examples which are still regular near infinity,
including a $5^27^2$-triangulation.

Before sketching Dress's argument further, we note that the Holonomy Theorem
gives an alternative proof of the three most important cases.
(These are the only cases arising in simplicial triangulations---and
indeed the only cases considered by Dress~\cite{Dre}, although
we will see below that his original argument applies equally well
to the cases of $2,\!10$- and $1,\!11$-triangulations.)
\begin{cor}
A $5,\!7$-, $4,\!8$- or $3,\!9$-triangulation of the plane
cannot be isomorphic to the regular triangulation near infinity.
\end{cor}
\begin{proof}
As above, given a triangulation that is regular near infinity,
we can glue opposite sides of a large rectangle to produce a triangulation
of the torus, with the same two exceptional vertices.
The $5,\!7$ case is then ruled out immediately by 
Theorem~\ref{thm:57}.

For a $4,\!8$- or $3,\!9$-triangulation,
consider the holonomy of the equilateral
metric on the triangulated torus.  The sides of the rectangle
generate the fundamental group of the (unpunctured) torus,
but since they lie in the regular background triangulation,
they have no rotational holonomy.
Thus the holonomy group is generated just by loops around
the cone points, contradicting the Holonomy Theorem.
\qed\end{proof}

The proof sketched by Dress can be understood as an argument that
a triangulation with two exceptional vertices has nonzero Burgers vector,
while a triangulation regular near infinity must have zero Burgers vector.
To make this precise, we now propose a mathematical interpretation
of the Burgers vector in terms of the developing map.

\begin{dfn}
Given any euclidean cone metric on an oriented surface~$M$,
a \emph{developing map} is an oriented local isometry mapping
the universal cover of~$M^o$ to the euclidean plane.
Pre-composing a developing map with a deck transformation of the covering
gives a new developing map.
But any two developing maps for the same cone metric differ
by (post-composition with) a euclidean motion in the plane.
Identifying the group of deck
transformations with the fundamental group $\pi_1(M^o,x)$,
where~$x$ is any basepoint, we thus get a homomorphism
$$\hat h: \pi_1(M^o,x)\to SE_2,$$
where $SE_2=\R^2\rtimes SO_2$ is the euclidean group.
We call~$\hat h$ the \emph{holonomy} of the developing map.
Its rotational part (that is, its composition with
the projection $SE_2\to SO_2$) is the metric
holonomy~$h$ discussed above.
\end{dfn}

For the equilateral metric arising from a triangulation
of~$M$, the developing map sends each triangle linearly to one
in the regular triangulation of the plane.  Thus its holonomy
lies in the symmetry group of the triangular lattice.

The image of a closed path~$\gamma$ in~$M^o$
is a locally isometric path in the plane, with the
same geodesic curvature.  In particular, if~$\gamma$
follows edges of the triangulation through regular vertices,
then its image follows edges of the triangular lattice
in the plane, making the same turn at each corresponding
vertex.  This mimics the usual crystallographic definition
of the Burgers vector.  In particular, suppose~$\gamma$ is a loop in~$M^o$
bounding a disk in $M$ in which the cone points have total curvature zero.
Then $h(\gamma)=0$, so $\hat h(\gamma)$ is a translation,
called the \emph{Burgers vector}.

We can now flesh out Dress's argument.
\begin{thm}[Dress~\cite{Dre}]
A triangulation of the plane with exactly two exceptional vertices
cannot be isomorphic to the regular triangulation near infinity.
\end{thm}
\begin{proof}
Given a triangulation of the plane isomorphic to the regular one
near infinity, let~$\gamma$ be a loop around all the exceptional vertices.
Taking a representative of the homotopy class near infinity, we see
$\hat h(\gamma)=0$.
If there are just two opposite cone points~$p_\pm$, then~$\gamma$
can also be written as a loop around~$p_+$ followed by one around $p_-$.
The holonomy of each individual loop is a rotation around the image of
the cone point under the developing map.  Of course each cone point
has many images, but the images of $p_\pm$ we care about are
are distinct points in the plane---indeed separated by the
same distance as the original points $p_\pm$.  Then we are done, since
the two opposite rotations around distinct centers compose to give
a nonzero translation $\hat h(\gamma)$, the Burgers vector.
\qed\end{proof}

We have seen above that the Holonomy Theorem gives an alternative proof
of most cases of Dress's theorem.  Physicists have often speculated
that, conversely, arguments like those of Dress could be used
to prove Theorem~\ref{thm:57}, that is, to rule out of $5,\!7$-triangulation
of a torus.  More precisely, the argument is supposed to be as follows:
Given a triangulation of the torus with two exceptional vertices,
choose standard generators~$\alpha$ and~$\beta$ for the fundamental
group of the torus.  The loop
$$\gamma := \alpha\cdot\beta\cdot\alpha^{-1}\cdot\beta^{-1}$$
is a loop around both cone points, so as in Dress's argument, its
holonomy must be a nonzero translation, the product of two opposite
rotations about different centers.  The contradiction is now supposed
to arise from the fact that $\gamma=[\alpha,\beta]$ is a commutator
and thus should give zero Burgers vector.  This argument would work if
the holonomy of~$\alpha$ and~$\beta$ were purely translational.
Instead, we see that this argument proves at least one of them has
nontrivial rotational holonomy so that the commutator in $SE_2$ is
nontrivial, as in Figure~\ref{fig:48}.  This is essentially a weaker
variant of the Holonomy Theorem, which also says something
about the amount of rotation; it does not seem that
Theorems~\ref{thm:57}--\ref{thm:26quad} would follow from this weaker
variant.

\begin{figure}
\centering
\includegraphics[height=5cm]{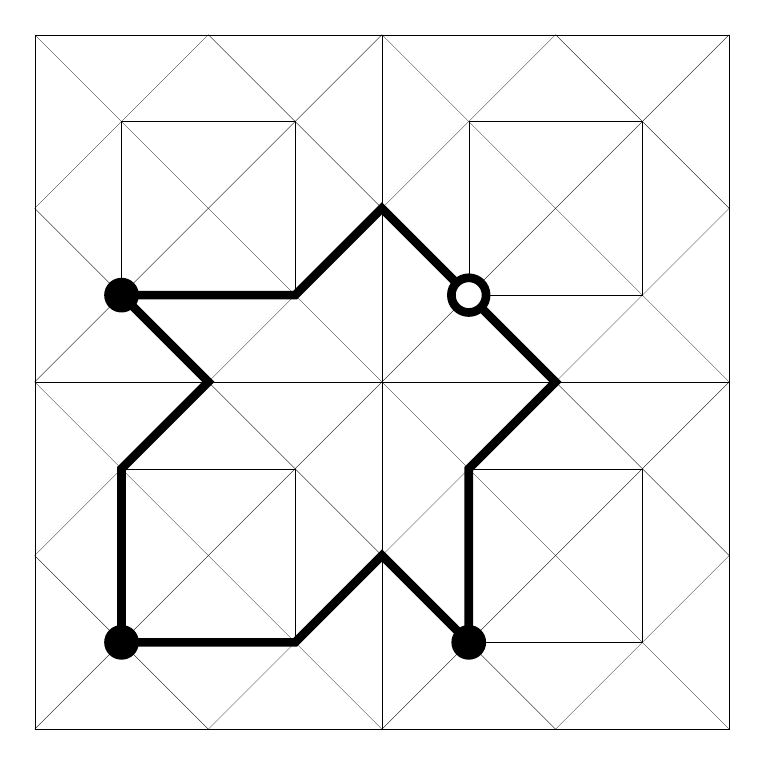}
\hspace{2cm}
\includegraphics[height=62mm]{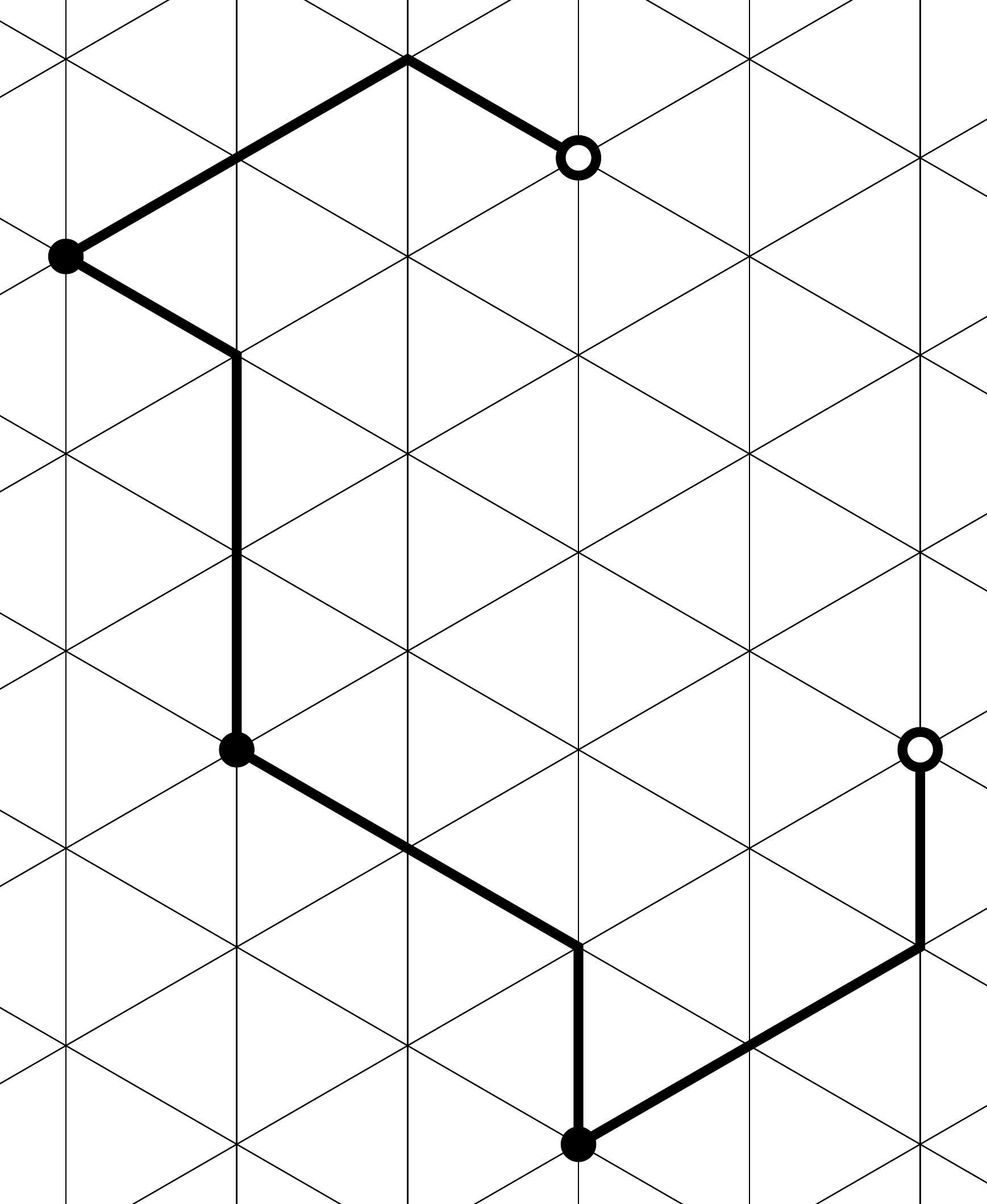}
\caption[Burgers vector for the 48-triangulation]
{The $2$-fold refinement (left) of the $4,\!8$-triangulation
of Figure~\ref{fig:examples}(a)
has a fundamental domain whose boundary curve $\gamma$
(shown as the heavier line)
only passes through regular vertices.  If we trace the same steps in
a regular hexagonal lattice, we get a nonclosed path (right);
the difference of its endpoints is the nonzero Burgers vector of~$\gamma$.
Note that opposite arcs in the boundary of the fundamental
domain (which get glued to each other to form the torus)
are no longer parallel to each other when developed into
the regular lattice; they are rotated by $\pi/3$, corresponding
to the holonomy $H=C_6$ guaranteed by Corollary~\ref{cor:48holo}.
The intuition that a periodic structure should have zero
Burgers vector thus fails.  Indeed the rotational holonomy
shows that $\hat h$ is not a translation for the nontrivial loops
$\alpha$ and $\beta$ on the torus, so their Burgers
vectors aren't even defined.}
\label{fig:48}
\end{figure}

\section{Examples of non-toroidal graphs}
\label{sec:graphs}
A graph is called \emph{toroidal} if it can be embedded in the
torus. As implied by the Robertson--Seymour theorem, toroidal graphs
are characterized by a finite set of forbidden minors.
(See~\cite{GMC09} for more background information.)

Theorems~\ref{thm:57}--\ref{thm:24} provide 
an infinite family of non-toroidal graphs. Recall that the {\it girth} of
a graph is the length of the shortest cycle in it.  
The girth is at least $3$ if and only if the graph has no
loops or multiple edges.

\begin{cor} \label{cor:57_Embed}
If a graph~$G$ satisfies any one of the following sets of assumptions,
then it cannot be embedded in the torus:

\begin{enumerate}[(a)]
\item All vertices of~$G$ have degree~$6$, except for one of degree~$5$
and one of degree~$7$, 
and $G$ has girth at least~$3$.

\item All vertices of~$G$ have degree~$4$, except for one of degree~$3$
and one of degree~$5$, and $G$ has girth at least~$4$.

\item All vertices of~$G$ have degree~$3$, except for one of degree~$2$
and one of degree~$4$, and $G$ is bipartite with girth at least~$6$.
\end{enumerate}
\end{cor}

\begin{proof}
  Our assumptions mean that the average vertex degree $k$ is~$6$, $4$ or~$3$,
  respectively, while the girth is at least $\bar k := 2k/(k-2)$.
  If $G$ is embedded in a torus, then each face
  has at least $\bar k$ sides. (Here, a face is a connected component of
  the complement of $G$. A priori, it need not be a
  topological disk, nor does its boundary have to be connected.) By
  double counting the edges in two ways as before, we obtain
  the inequality
  $$0=\chi=V-E+F\leq 2E\Bigl(\frac{1}{k}-\frac{1}{2}+\frac{1}{\bar k}\Bigr).$$
  On the other hand, we have $V-E+F\geq 0$, with equality if and only if all
  faces are topological disks. It follows that all faces are
  topological disks with $\bar k$ sides.  But then we have a triangulation,
  quadrangulation or bipartite hexangulation of the torus.  Such maps,
  with the given vertex degrees, have been ruled out by
  Theorems~\ref{thm:57}--\ref{thm:24}.
\qed\end{proof}

The following procedure produces graphs to which
Corollary~\ref{cor:57_Embed} is applicable.
Let $G$ be any $6$-regular graph without loops or multiple edges
and with at least $8$ vertices.
Choose an edge~$ij$ of~$G$, and let~$k$ be any vertex not adjacent to~$i$.
Remove the edge~$ij$ and insert an edge~$ik$.
Vertices~$j$ and~$k$ now have degrees~$5$ and~$7$, respectively,
while all the other vertices still have degree~$6$.

The same can be done with $4$- and $3$-regular graphs.
We start with a graph of the required girth (again, the
regular tessellations of the torus provide examples)
and then simply choose vertex~$k$ in the above procedure far enough from~$i$.
Figures~\ref{fig:no-torus-graphs} and~\ref{fig:girth-6}
show graphs constructed by this procedure;
they are not embeddable in the torus.
\begin{figure}
  \centering \includegraphics[width=.8\textwidth]{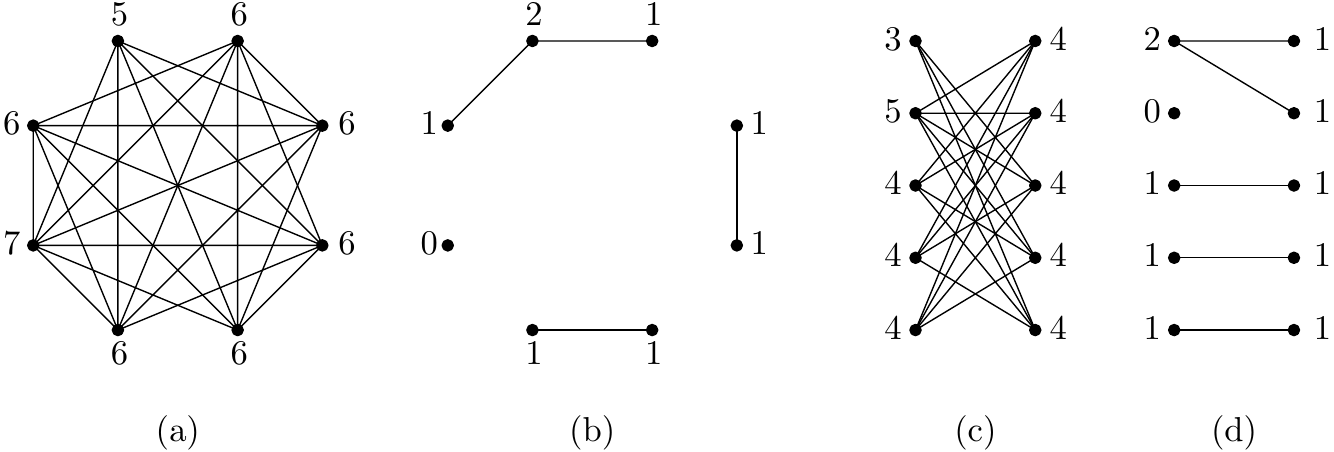}
  \caption{(a) A graph with vertex degrees $5\,6^6 7$, and (b)
  its complement;
(c) a graph of girth 4 with vertex degrees $3\,4^8 5$, and (d) its complement.
These are the smallest examples of their respective
classes of graphs.}
  \label{fig:no-torus-graphs}
\end{figure}
\begin{figure}
  \centering \includegraphics[scale=.9]{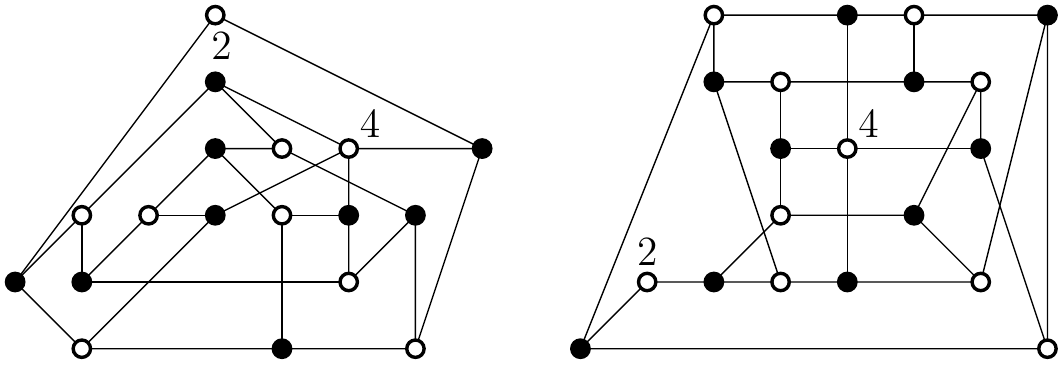}
  \caption{Bipartite graphs with girth 6 and vertex degrees $2\,3^k 4$.}
  \label{fig:girth-6}
\end{figure}

Any $5,\!7$-triangulation of the Klein bottle provides another
example of such a graph, as long as it has girth~$3$.
The one in Figure~\ref{fig:examples57}(b)
is too small---its 1-skeleton has girth only~$2$---but 
a refinement of it would work.

\begin{acknowledgements}
We would like to thank Ken Stephenson for 
mentioning this problem~\cite[Problem~13, p.~694]{DDG-problems}
and for the data for Figure~\ref{fig:KS57ex};
Ulrich Brehm, Gunnar Brinkmann and G\"unter M.~Ziegler
for pointing us to relevant literature;
and Frank Lutz for helpful discussions.
I.~Izmestiev and J.~M.~Sullivan were partially supported by the DFG Research
Group 565 ``Polyhedral Surfaces''. R.~Kusner was supported in part by
NSF Grant DMS-0076085. 
B.~Springborn was partially supported by
the DFG Research Center \textsc{Matheon}.
B.~Springborn and J.~M.~Sullivan were partially supported by DFG
SFB/Transregio 109 ``Discretization in Geometry and Dynamics''.
\end{acknowledgements}

\bibliography{torus}
\end{document}